\documentclass[a4paper]{amsart}
\usepackage[T1]{fontenc}
\usepackage{lmodern}
\usepackage{amssymb}
\usepackage[all]{xy}
\usepackage{enumitem}
\usepackage{nicefrac,stmaryrd}
\usepackage{mathtools}
\usepackage[ngerman, british]{babel}
\usepackage{microtype}

\usepackage[pdftitle={Comparison of KE-Theory and KK-Theory},
  pdfauthor={Ralf Meyer},
  pdfsubject={Mathematics}
]{hyperref}
\usepackage[lite]{amsrefs}
\newcommand*{\MRref}[2]{ \href{http://www.ams.org/mathscinet-getitem?mr=#1}{MR \textbf{#1}}}

\renewcommand{\PrintDOI}[1]{\href{http://dx.doi.org/\detokenize{#1}}{doi: \detokenize{#1}}%
  \IfEmptyBibField{pages}{, (to appear in print)}{}}

\mathtoolsset{mathic}

\newcommand*{\KK}{\mathrm{KK}}
\newcommand*{\E}{\mathrm{E}}
\newcommand*{\KE}{\mathrm{KE}}
\newcommand*{\K}{\mathrm{K}}
\newcommand*{\Cst}{\mathrm{C}^*}
\newcommand*{\op}{\mathrm{op}}

\newcommand*{\norm}[1]{\lVert #1\rVert} % norm

\newcommand*{\Cont}{\mathrm C} % continuous functions
\newcommand*{\Hils}{\mathcal H}% Hilbert space
\newcommand*{\Hilm}{\mathcal E}% Hilbert module

\newcommand*{\ID}{\mathrm{id}}

\newcommand*{\hot}{\mathbin{\hat\otimes}}%graded tensor product

\newcommand*{\R}{\mathbb R}
\newcommand*{\Z}{\mathbb Z}
\newcommand*{\N}{\mathbb N}
\newcommand*{\Bound}{\mathcal L}% adjointable operators
\newcommand*{\Comp}{\mathcal K} % compact operators
\newcommand*{\Null}{\mathcal I} % asymptotically zero operators
\newcommand*{\Pcomp}{\mathcal C}% locally compact operators

\newcommand*{\nb}{\nobreakdash}  % no break after this hyphen

\newcommand*{\Star}{\texorpdfstring{$^*$\nb-}{*-}}

\newcommand*{\defeq}{\mathrel{\vcentcolon=}}

\numberwithin{equation}{section}

\theoremstyle{plain}
\newtheorem{theorem}[equation]{Theorem}
\newtheorem{proposition}[equation]{Proposition}
\newtheorem{lemma}[equation]{Lemma}

\theoremstyle{definition}

\theoremstyle{remark}
\newtheorem{remark}[equation]{Remark}

\begin{document}

\title{Comparison of KE-Theory and KK-Theory}

\author{Ralf Meyer}
\address{Mathematisches Institut\\
         Georg-August-Universit\"at G\"ottingen\\
         Bunsenstr.\ 3--5\\
         37073 G\"ottingen\\
         Germany
}
\email{rmeyer2@uni-goettingen.de}


\begin{abstract}
  We show that the character from the bivariant \(\K\)\nb-theory
  \(\KE^G\) introduced by Dumitra\c{s}cu to \(\E^G\) factors through
  Kasparov's \(\KK^G\) for any locally compact group~\(G\).  Hence
  \(\KE^G\) contains \(\KK^G\) as a direct summand.
\end{abstract}

\maketitle

\section{Introduction}
\label{sec:intro}

\(\K\)\nb-theory may be generalised in several ways to a bivariant
theory.  One such bivariant \(\K\)\nb-theory is Kasparov's \(\KK\)
(see~\cite{Kasparov:Operator_K}), another is the \(\E\)\nb-theory of
Connes and Higson~\cite{Connes-Higson:Deformations}.  Both theories
have equivariant versions with respect to second-countable locally
compact groups (see \cites{Kasparov:Novikov,
  Guentner-Higson-Trout:Equivariant_E}).  These theories are related
by a natural transformation \(\KK^G(A,B)\to\E^G(A,B)\) because of the
universal property of~\(\KK^G\).

Dumitra\c{s}cu defines another equivariant bivariant \(\K\)\nb-theory
\(\KE^G(A,B)\) in his thesis~\cite{Dumitrascu:Thesis}, which has the
same formal properties as \(\KK^G\) and~\(\E^G\); in particular, it
has an analogue of the Kasparov product and the exterior product.  He
also constructs an explicit natural transformation
\[
\KK^G(A,B)\to\KE^G(A,B)\to\E^G(A,B).
\]
Hence this also makes the transformation \(\KK^G\to\E^G\) explicit.

The construction of a new bivariant \(\K\)\nb-theory is always laden
with technical difficulties, especially the construction of a product.
\(\KK\)\nb-theory and \(\E\)\nb-theory involve different
technicalities, and \(\KE\)\nb-theory needs a share of both kinds of
technicalities.  When I was asked to referee the
article~\cite{Dumitrascu:KE} by Dumitra\c{s}cu, I therefore wanted to
clarify whether~\(\KE^G\) is really a new theory or equivalent to
\(\KK^G\) or~\(\E^G\).  I expected~\(\KE^G\) to be equivalent to
either \(\KK^G\) or~\(\E^G\).  I came up quickly with a sketch of an
argument why~\(\KE^G\) should be equivalent to~\(\KK^G\), which I
communicated to Dumitra\c{s}cu, asking him whether he could complete
this sketch to a full proof.  After a while it became clear that I had
to complete this argument myself, which resulted in this note.  Its
purpose is the following theorem:

\begin{theorem}
  \label{the:KE_equals_KK}
  Let~\(G\) be a second countable locally compact group and let
  \(A\) and~\(B\) be separable \(G\)\nb-\(\Cst\)-algebras.  The
  natural map \(\KE^G(A,B)\to\E^G(A,B)\) factors through a map
  \(\KE^G(A,B)\to\KK^G(A,B)\).
\end{theorem}

I still expect \(\KE^G(A,B)\cong\KK^G(A,B)\), but I do not know how to
prove this.  A transformation \(\KE^G\to\KK^G\) seems enough for
applications.  It shows that a computation in~\(\KE^G\) gives results
in~\(\KK^G\).  For instance, if~\(G\) has the analogue of a
\(\gamma\)\nb-element in~\(\KE^G\) or if \(\gamma=1\) in~\(\KE^G\),
then the same follows in~\(\KK^G\).

In Section~\ref{sec:def_KE}, we recall Dumitra\c{s}cu's definition of
cycles for \(\KE^G(A,B)\) and show that we may strengthen his
conditions slightly without changing the set of homotopy classes.  In
Section~\ref{sec:cp_am_from_KE}, we show how to get completely
positive equivariant asymptotic morphisms from the \(\KE^G\)-cycles
satisfying our stronger conditions.

\section{The definition of KE-theory}
\label{sec:def_KE}

Throughout this article, \(G\) is a second countable, locally compact
topological group; \(A\) and~\(B\) are separable \(\Cst\)\nb-algebras
with continuous actions of \(G\times\Z/2\).  An action of
\(G\times\Z/2\) is the same as a \(\Z/2\)\nb-grading together with an
action of~\(G\) by grading-preserving automorphisms; we will
frequently combine a \(\Z/2\)\nb-grading and a \(G\)\nb-action in this
way.

We first recall the definition of \(\KE^G(A,B)\).

Let \(L\defeq [1,\infty)\) and \(BL\defeq \Cont_0(L,B)\).  A
\emph{continuous field of \(G\times\Z/2\)-\((A,B)\)-modules} is a
countably generated, \(G\times\Z/2\)\nb-equivariant Hilbert
\(BL\)-module~\(\Hilm\) with a \(G\times\Z/2\)-equivariant
\Star{}homomorphism \(\varphi\colon A\to\Bound(\Hilm)\), where
\(\Bound(\Hilm)\) denotes the \(\Cst\)\nb-algebra of adjointable
operators on~\(\Hilm\) with its canonical action of \(G\times\Z/2\).

Using the evaluation homomorphisms \(BL\to B\), \(f\mapsto f(t)\), we
may view~\(\Hilm\) as a family of \(G\times\Z/2\)-equivariant Hilbert
\(B\)\nb-modules~\(\Hilm_t\); an operator \(x\in\Bound(\Hilm)\) is
completely determined by a family of operators
\(x_t\in\Bound(\Hilm_t)\).  Besides the ideal \(\Comp(\Hilm)\) of
compact operators on~\(\Hilm\), we need the two ideals
\begin{align*}
  \Pcomp(\Hilm) &\defeq \{x\in\Bound(\Hilm) \mid xf\in\Comp(\Hilm)
  \text{ for all } f\in\Cont_0(L)\},\\
  \Null(\Hilm) &\defeq \{x\in\Bound(\Hilm) \mid \lim_{t\to\infty} \norm{x_t}=0\}.
\end{align*}
We have \(\Pcomp(\Hilm)\cap\Null(\Hilm)=\Comp(\Hilm)\).

A \emph{cycle for \(\KE^G(A,B)\)} is a pair \((\Hilm,F)\),
where~\(\Hilm\) is a continuous field of
\(G\times\Z/2\)-\((A,B)\)-modules and~\(F\) is an odd, adjointable
operator on~\(\Hilm\) that satisfies the following conditions (for all
\(a\in A\), \(g\in G\)):
\begin{description}
\item[aKm1] \((F-F^*)\varphi(a)\in\Null(\Hilm)\) for all \(a\in A\);
\item[aKm2] \([F,\varphi(a)]\in\Null(\Hilm)\) for all \(a\in A\);
\item[aKm3] \(\varphi(a)(F^2-1)\varphi(a)^*\ge0\) modulo
  \(\Pcomp(\Hilm)+\Null(\Hilm)\) for all \(a\in A\);
\item[aKm4] \((gF-F)\varphi(a)\in\Null(\Hilm)\) for all \(a\in A\),
  \(g\in G\).
\end{description}

Later, we shall meet the following strengthenings of these conditions:
\begin{description}
\item[aKm1s] \(F=F^*\);
\item[aKm3s] \(\norm{F}\le1\) and \((1-F^2)\varphi(a)\in\Pcomp(\Hilm)\)
  for all \(a\in A\);
\item[aKm4s] \(gF=F\) for all \(g\in G\).
\end{description}

Cycles for \(\KE^G(A,\Cont([0,1],B))\) are called \emph{homotopies} of
cycles.  We define \(\KE^G(A,B)\) as the set of homotopy classes of
cycles for \(\KE^G(A,B)\).

\begin{lemma}
  \label{lem:strengthen_KE_cycle}
  Any cycle for \(\KE^G(A,B)\) is homotopic to one that satisfies
  \textup{(aKm1s)} and \textup{(aKm3s)}.  If two cycles satisfying
  \textup{(aKm1s)} and \textup{(aKm3s)} are homotopic, they are
  homotopic via a homotopy that satisfies \textup{(aKm1s)} and
  \textup{(aKm3s)}.
\end{lemma}

We will treat condition~(aKm4s) below in
Lemma~\ref{lem:equivariant_KE_cycle}.

\begin{proof}
  Let \((\Hilm,F)\) be a cycle for \(\KE^G(A,B)\).  Then \((F+F^*)/2\)
  is a small perturbation of~\(F\) and hence gives a homotopic cycle
  (see \cite{Dumitrascu:KE}*{Corollary 2.25}) satisfying
  \(F=F^*\).

  Now assume \(F=F^*\) and~(aKm2); then
  \begin{multline*}
    \varphi(a)(1-F^2)_+ \varphi(a)^*\cdot \varphi(a)(1-F^2)_- \varphi(a)^*
    \\\equiv \varphi(a) \varphi(a)^* \varphi(a)(1-F^2)_+(1-F^2)_- \varphi(a)^*
    \equiv 0 \bmod \Null(\Hilm).
  \end{multline*}
  Hence \(\varphi(a) (F^2-1)_\pm \varphi(a)^*\) are the positive and
  negative parts of \(\varphi(a) (F^2-1)\varphi(a)^*\) in
  \(\Bound(\Hilm)/\Null(\Hilm)\).  As a result, (aKm3) is equivalent
  to \(\varphi(a)\cdot (F^2-1)_-\varphi(a)^*\in \Pcomp(\Hilm)+\Null(\Hilm)\)
  for all \(a\in A\).

  Define \(\chi\colon \R\to[-1,1]\) by \(\chi(t)\defeq -1\) for
  \(t\le-1\), \(\chi(t)\defeq t\) for \(-1\le t\le1\), and
  \(\chi(t)\defeq 1\) for \(t\ge1\).  Then \(\norm{\chi(F)}\le1\) and
  \(\chi(F)^2-1=(F^2-1)_-\).  The reformulation of~(aKm3) in the
  previous paragraph shows that
  \((\Hilm,\chi(F))\) is again a cycle for \(\KE^G(A,B)\) and that the
  linear path \((\Hilm,sF+(1-s)\chi(F))\) is a homotopy of cycles.
  Thus any cycle is homotopic to one with \(F=F^*\) and
  \(\norm{F}\le1\).

  Next we adapt the standard trick to achieve \(F^2=1\) for
  \(\KK\)-cycles.  Let \(\Hilm_2\defeq\Hilm\oplus\Hilm^\op\),
  where~\(\op\) denotes the opposite \(\Z/2\)\nb-grading.  Let~\(A\)
  act on~\(\Hilm_2\) by \(\varphi_2\defeq \varphi\oplus0\).  For
  \(s\in[0,1]\), let
  \[
  F_{2s} \defeq
  \begin{pmatrix}
    F&s\sqrt{1-u^2}\sqrt{1-F^2}\\
    s\sqrt{1-F^2}\sqrt{1-u^2}&-F
  \end{pmatrix},
  \]
  where \(u\in\Bound(\Hilm)^{(0)}\) is an even operator as in
  Lemma~\cite{Dumitrascu:KE}*{Lemma 2.35}; that is,
  \(u\in\Pcomp(\Hilm)\), \([u,F]\in\Null(\Hilm)\),
  \([u,\varphi(a)]\in\Null(\Hilm)\) for all \(a\in A\),
  \((1-u^2)(\varphi(a) (F^2-1)\varphi(a)^*)_-\in\Null(\Hilm)\) for all
  \(a\in A\), and \(gu-u\in\Null(\Hilm)\) for all \(g\in G\).  Since
  \(u\in\Pcomp(\Hilm)\) and
  \(\Pcomp(\Hilm)\cap\Null(\Hilm)=\Comp(\Hilm)\), we even have
  \([u,F]\in\Comp(\Hilm)\), \([u,\varphi(a)]\in\Comp(\Hilm)\) for all
  \(a\in A\), and \(gu-u\in\Comp(\Hilm)\) for all \(g\in G\).  Since
  we already achieved \(\norm{F}\le1\), we also have
  \((1-u^2)\varphi(a) (1-F^2)\varphi(a)^*\in\Null(\Hilm)\), hence
  \((1-u^2)\varphi(aa^*)(1-F^2)\in\Null(\Hilm)\).  This is equivalent
  to \((1-u^2)\varphi(a)(1-F^2)\in\Null(\Hilm)\) for all \(a\in A\)
  because elements of the form \(aa^*\) span~\(A\).

  The set of \(f\in \Cont[0,1]\) with \(f(1-u^2)\varphi(a)(1-F^2)\in
  \Null(\Hilm)\) for all \(a\in A\) is a closed ideal because
  \(\Null(\Hilm)\) is a closed ideal.  Since \(1-u^2\) and
  \(\sqrt{1-u^2}\) generate the same closed ideal in \(\Cont[0,1]\),
  namely, the ideal of functions vanishing at~\(1\), our condition is
  equivalent to \(\sqrt{1-u^2}\varphi(a)(1-F^2)\in\Null(\Hilm)\) for
  all \(a\in A\).  We may do the same to~\(F\), so our condition is
  also equivalent to \(\sqrt{1-u^2}\varphi(a)\sqrt{1-F^2} \in
  \Null(\Hilm)\) for all \(a\in A\).  Moreover, we may change the
  order of the three factors here arbitrarily.  Therefore,
  \([F_{2s},\varphi_2(a)]\in\Null(\Hilm_2)\) for all \(a\in A\).
  Furthermore, \([u,F]\in\Comp(\Hilm)\) implies
  \begin{align*}
    & \phantom{{}={}} (1-F_{2s}^2)\varphi(a)
    % \\&= \begin{pmatrix}
    %   1-F^2-s^2\sqrt{1-u^2}(1-F^2)\sqrt{1-u^2}&
    %   -sF\sqrt{1-u^2}\sqrt{1-F^2}+ s\sqrt{1-u^2}\sqrt{1-F^2}F\\
    %   -s\sqrt{1-F^2}\sqrt{1-u^2}F + sF\sqrt{1-F^2}\sqrt{1-u^2}&
    %   1-s^2\sqrt{1-F^2}(1-u^2)\sqrt{1-F^2}-F^2
    % \end{pmatrix}\varphi(a)
    \\&\equiv
    \begin{pmatrix}
      (1-F^2)(1-s^2+s^2u^2)& 0\\
      0& (1-F^2)(1-s^2+s^2u^2)
    \end{pmatrix}\varphi(a)
    \bmod \Comp(\Hilm_2).
  \end{align*}
  Hence \((\Hilm_2,F_{2s})\) is a homotopy of cycles for
  \(\KE^G(A,B)\).  For \(s=0\), \((\Hilm_2,F_{20})\) is a direct sum
  of \((\Hilm,F)\) with a degenerate cycle and hence homotopic to
  \((\Hilm,F)\).  Thus \((\Hilm,F)\) is homotopic to
  \((\Hilm_2,F_{21})\).  The diagonal entries of \(1-F_{21}^2\) are
  \((1-F^2)u^2\), which belongs to~\(\Pcomp(\Hilm)\) because
  \(u\in\Pcomp(\Hilm)\).  Hence \(1-F_{21}^2\in\Pcomp(\Hilm_2)\).
  Thus any cycle for \(\KE^G(A,B)\) is homotopic to one satisfying
  (aKm1s) and (aKm3s).

  If we already have \(F=F^*\) and \(\norm{F}\le1\), then the canonical
  homotopy from \(F\) to \(\chi((F+F^*)/2)\) is constant.  And if also
  \(1-F^2\in\Pcomp(\Hilm)\), then the homotopy~\(F_{2s}\) constructed
  above satisfies \(1-F_{2s}^2\in\Pcomp(\Hilm_2)\) for any choice
  of~\(u\).  If two cycles \((\Hilm_1,F_1)\) and \((\Hilm_2,F_2)\)
  satisfying (aKm1s) and (aKm3s) are homotopic, then we may apply the
  modifications above to a homotopy between them; this provides a
  homotopy between their modifications that satisfies (aKm1s) and
  (aKm3s); since the canonical homotopies from \((\Hilm_1,F_1)\) and
  \((\Hilm_2,F_2)\) to their modifications also satisfy (aKm1s)
  and~(aKm3s), we get a homotopy from \((\Hilm_1,F_1)\) to
  \((\Hilm_2,F_2)\) satisfying (aKm1s) and (aKm3s).  Hence restricting
  to cycles satisfying (aKm1s) and (aKm3s) does not change
  \(\KE^G(A,B)\).
\end{proof}

The \emph{standard \(G\times\Z/2\)\nb-equivariant Hilbert
  \(B\)\nb-module} is
\[
\Hils=\Hils_B\defeq L^2(G\times\Z/2)\otimes \ell^2(\N)\otimes B.
\]

\begin{lemma}
  \label{lem:equivariant_KE_cycle}
  We get the same group
  \(\KE^G(A\otimes\Comp(L^2G),B\otimes\Comp(L^2G))\) if we restrict
  attention to cycles for
  \(\KE^G(A\otimes\Comp(L^2G),B\otimes\Comp(L^2G))\) that satisfy
  \textup{(aKm1s)}, \textup{(aKm3s)} and \textup{(aKm4s)} and with
  underlying Hilbert module \(\Hilm=\Hils_{B\otimes\Comp(L^2 G)} L\),
  and homotopies between such cycles with the same properties.
\end{lemma}

\begin{proof}
  The main ideas below already appeared in~\cite{Meyer:Equivariant}
  and as Fell's trick in \cite{Dumitrascu:Thesis}*{Lemma 3.3.3}.
  Let~\(F_0\) be the canonical isomorphism
  \(\Hils_+\leftrightarrow\Hils_-\) and let \(\varphi_0=0\); this
  gives a degenerate cycle with underlying Hilbert module~\(\Hils L\).
  Hence any \(\KE^G(A,B)\)-cycle \((\Hilm,F)\) is equivalent to
  \((\Hilm\oplus\Hils L,F\oplus F_0)\).  Since~\(\Hilm\) must be
  countably generated, Kasparov's Stabilisation Theorem gives a
  \(G\)\nb-continuous, \(\Z/2\)\nb-equivariant unitary operator
  \(V\colon \Hilm\oplus\Hils L\to\Hils L\).  (Unless~\(G\) is compact,
  we cannot expect~\(V\) to be \(G\)\nb-equivariant.)

  Therefore, we get the same set of homotopy classes \(\KE^G(A,B)\) if
  we restrict attention to cycles~\((\Hilm,F)\) for which there is a
  \(G\)\nb-continuous, \(\Z/2\)\nb-grading preserving unitary
  \(V\colon \Hilm \to \Hils_B L\).  This may be combined with
  Lemma~\ref{lem:strengthen_KE_cycle}, that is, we get the same set of
  homotopy classes if we assume~\((\Hilm,F)\) to satisfy (aKm1s) and
  (aKm3s) and to have such a unitary~\(V\).  The unitary~\(V\) defines
  a \(G\times\Z/2\)-equivariant unitary
  \[
  V'\colon L^2(G,\Hilm) \to L^2(G,\Hils_B L),\qquad
  (V' f)(g) \defeq g (V(f(g))).
  \]
  By a similar formula, any \(F\in\Bound(\Hilm)\) defines a
  \(G\)\nb-equivariant adjointable operator~\(F'\) on
  \(L^2(G,\Hilm)\).  By \cite{Dumitrascu:KE}*{Theorem 3.21}, the
  exterior product map
  \[
  \KE^G(A,B)\to\KE^G(A\otimes\Comp(L^2G),B\otimes\Comp(L^2G)),\quad
  (\Hilm,F)\mapsto (\Hilm\otimes\Comp(L^2G),F\otimes1),
  \]
  is an isomorphism.  So any cycle for
  \(\KE^G(A\otimes\Comp(L^2G),B\otimes\Comp(L^2G))\) is homotopic to
  \((\Hilm\otimes\Comp(L^2G),F\otimes1)\) for some cycle \((\Hilm,F)\)
  for \(\KE^G(A,B)\) with (aKm1s) and (aKm3s) and a unitary \(V\colon
  \Hilm \to \Hils_B L\) as above; and if two such cycles are
  homotopic, there is a homotopy of the same form.

  As a Hilbert module over itself, \(\Comp(L^2G)\cong L^2G\otimes (L^2
  G)^*\), where~\((L^2G)^*\) is viewed as a Hilbert
  \(\Comp(L^2G)\)-module.  Hence~\(V'\) induces a
  \(G\times\Z/2\)-equivariant unitary \(\Hilm\otimes\Comp(L^2G)\to
  \Hils_B L\otimes\Comp(L^2G)=\Hils_{B\otimes\Comp(L^2G)} L\),
  and~\(F'\) induces a \(G\)\nb-equivariant odd operator on
  \(\Hilm\otimes\Comp(L^2G)\).  Since \(gF-F\in\Null(\Hilm)\), \(F'\)
  is a small perturbation of \(F\otimes 1\).  Thus we get the same
  group \(\KE^G(A\otimes\Comp(L^2G), B\otimes\Comp(L^2G))\) if we use
  only those cycles and homotopies that satisfy (aKm1s), (aKm3s) and
  (aKm4s) and have underlying Hilbert module
  \(\Hilm=\Hils_{B\otimes\Comp(L^2 G)} L\).
\end{proof}

For the passage from \(\KE^G\) to \(\E^G\), it is harmless to
stabilise the \(\Cst\)\nb-algebras \(A\) and~\(B\).  Hence
Lemma~\ref{lem:equivariant_KE_cycle} says that it is essentially no
loss of generality to restrict attention to those cycles for \(\KE^G\)
that satisfy the stronger assumptions (aKm1s), (aKm3s) and (aKm4s).
Furthermore, we may assume that \(\Hilm=\Hils_B L\) is the constant
family with fibre the standard \(G\)\nb-equivariant Hilbert
\(B\)\nb-module~\(\Hils_B\).

\begin{remark}
  \label{rem:KK_in_KE_even_better}
  If a cycle for \(\KE^G(A,B)\) is in the image of \(\KK^G(A,B)\),
  then it satisfies more than (aKm2), namely, \([F,\varphi(a)]\in
  \Comp(\Hilm)\) for all \(a\in A\).  If \(\KK^G\) and~\(\KE^G\)
  were equivalent, then any cycle for~\(\KE^G\) would be homotopic
  to one with this extra property.  I do not know, however, how to
  prove this.
\end{remark}

\section{Constructing asymptotic morphisms from KE-cycles}
\label{sec:cp_am_from_KE}

Let \(S \defeq \Cont_0((-1,1))\) with the \(\Z/2\)-grading
automorphism \(\gamma f(x)=f(-x)\).  Dumitra\c{s}cu maps a cycle
\((\Hils_BL,\varphi,F)\) for \(\KE^G(A,B)\) to an asymptotic morphism
from \(S\hot A\) to \(\Comp(\Hils_B)\) in
\cite{Dumitrascu:KE}*{Section 4.1}, as follows.  Since \(\Null(\Hils_B
L)\cap\Pcomp(\Hils_BL) = \Comp(\Hils_BL)\) and
\([F,\varphi(A)]\subseteq\Null(\Hils_BL)\) by (aKm2), the images of
\(A\) and~\(F\) in \(\Pcomp(\Hils_BL)/\Comp(\Hils_BL)\) commute.
Hence there is a unique \Star{}homomorphism
\[
\Xi\colon S \hot A\to \Pcomp(\Hils_BL)/\Comp(\Hils_BL)
\]
with \(\Xi(h\otimes a) = h(F)\varphi(a)\) for all \(h\in S\), \(a\in
A\) (this works for the maximal \(\Cst\)\nb-norm, which is the only
\(\Cst\)\nb-norm here because~\(S\) is nuclear).  We may lift~\(\Xi\)
to a map (of sets) \(\bar\Xi\colon S \hot A\to \Pcomp(\Hils_BL)\),
which we may view as a family of maps \(\bar\Xi_t\colon S\hot A\to
\Comp(\Hils_B)\), \(t\in L\).  These maps~\(\bar\Xi_t\) form an
asymptotic morphism.  This is used in \cite{Dumitrascu:KE}*{Section
  4.1} to construct a functor \(\KE^G\to \E^G\).

For cycles with extra properties as in
Lemma~\ref{lem:equivariant_KE_cycle}, we are going to construct a
completely positive, contractive and \(G\times\Z/2\)\nb-equivariant
choice for~\(\bar\Xi\) in a natural way.  Using Thomsen's picture
for~\(\KK^G\), this will give a functor \(\KE^G\to\KE^G\), by
essentially the same arguments as in~\cite{Dumitrascu:KE}.

First we approximate the identity map on~\(S\) by
\(\Z/2\)\nb-equivariant, completely positive contractions of finite
rank.  Let \(n\in\N\).  Let \(I_n
\defeq\{-2^n+1,-2^n+2,\ldots,2^n-1\}\).  For \(k\in I_n\), define
\(\psi_{n,k}\in S\) by
\[
\psi_{n,k}(x) \defeq
\begin{cases}
  \sqrt{2^n x-(k-1)}&\text{for }k-1\le 2^n x\le k,\\
  \sqrt{k+1- 2^n x}&\text{for }k\le 2^n x\le k+1,\\
  0&\text{otherwise.}
\end{cases}
\]
Thus~\(\psi_{n,k}^2\) is the unique continuous, piecewise linear
function with singularities in \(2^{-n}\cdot\{k-1,k,k+1\}\) and
\(\psi_{n,k}^2(2^{-n}k)=1\) and \(\psi_{n,k}^2(2^{-n}l)=0\) for
\(k\neq l\).  We have \(\gamma(\psi_{n,k})=\psi_{n,-k}\) for all
\(k\in I_n\).  Define
\[
\Psi_n\colon S\to S,\qquad
f\mapsto \sum_{k\in I_n} f(2^{-n}k)\cdot \psi_{n,k}^2.
\]
Equivalently,
\begin{equation}
  \label{eq:Psi_n_concrete}
  \Psi_n f(2^{-n}(k+t)) =
  (1-t)\cdot f(2^{-n}k) + t\cdot f(2^{-n}(k+1))
\end{equation}
for \(k\in\{-2^n,-2^n+1,\ldots,2^n-1\}\), \(t\in[0,1]\), because
\(f(\pm1)=0\).

By construction, \(\Psi_n\) is a completely positive map of finite
rank.  It is grading-preserving because
\(\gamma(\psi_{n,k})=\psi_{n,-k}\), and contractive because
\(\sum_{k\in I_n} \psi_{n,k}^2\le1\).  We have \(\lim
\norm{f-\Psi_n(f)}_\infty=0\) for each \(f\in S\), and this holds
uniformly for~\(f\) in a compact subset of~\(S\) because all the
operators~\(\Psi_n\) are contractions.

Now let \(A\) and~\(B\) be \(\Z/2\)-graded \(\Cst\)\nb-algebras.
Let~\(\hot\) be the graded-commutative tensor product.  This is
functorial for grading-preserving completely positive contractions.
Hence we get a grading-preserving completely positive contraction
\(\Psi_n^A=\Psi_n\hot\ID_A\colon S\hot A\to S\hot A\).  The sequence
\(\Psi_n^A(f)\) converges in norm to~\(f\) for any \(f\in S\hot A\)
because~\(\Psi_n\) converges to~\(\ID_S\) uniformly on compact
subsets.

To make use of Lemma~\ref{lem:equivariant_KE_cycle}, we assume \(A =
A_0\otimes\Comp(L^2G)\) and \(B= B_0\otimes\Comp(L^2G)\) for some
\(\Z/2\)-graded \(\Cst\)\nb-algebras \(A_0\) and~\(B_0\).  Then we get
the same group \(\KE^G(A,B)\) if we use cycles and homotopies that
satisfy (aKm1s), (aKm2), (aKm3s) and (aKm4s), and where the underlying
family of Hilbert modules~\(\Hilm\) is the constant
family~\(\Hils_BL\) with the standard \(G\)\nb-equivariant Hilbert
\(B\)\nb-module~\(\Hils_B\) as its fibre.  (Actually, \(\Hils_B\) is
\(G\)\nb-equivariantly isomorphic to \((B^\infty)\oplus
(B^\infty)^\op\).)

Let \((\varphi,F)\) be such a special cycle for \(\KE^G(A,B)\).  That is,
\(\varphi\colon A\to\Bound(\Hils_BL)\) is a \(G\times\Z/2\)-equivariant
\Star{}homomorphism and \(F\in\Bound(\Hils_BL)\), such that
\(\gamma(F)=-F\), \(F=F^*\), \(\norm{F}\le 1\), \(g(F)=F\) for all
\(g\in G\), \(\lim_{t\to\infty} \norm{[F_t,\varphi_t(a)]}=0\) for all
\(a\in A\), and \((1-F^2)\varphi(a)\in\Pcomp(\Hils_BL)\).  Since 
\([(1-F^2)\varphi(a^*)]^* = \varphi(a)(1-F^2)\), it is equivalent to
require \((1-F^2)\varphi(a)\in\Pcomp(\Hils_BL)\) or
\(\varphi(a)(1-F^2)\in\Pcomp(\Hils_BL)\) for all \(a\in A\).
Furthermore, this implies \(h(F)\varphi(a)\in \Pcomp(\Hils_BL)\) and
\(\varphi(a)h(F)\in \Pcomp(\Hils_BL)\) for all \(h\in S\).

The next step is easier to write down for trivially graded~\(A\), so
we assume this for a moment to explain our idea.  Then \(S\hot A\cong
\Cont_0((-1,1),A)\).  Since \(\psi_{n,k}\in S\), we get
\(\psi_{n,k}(F) \varphi(a)\in\Pcomp(\Hils_BL)\) for all \(n\in\N\),
\(k\in I_n\), \(a\in A\).  Hence
\begin{equation}
  \label{eq:xi_n_formula}
  \xi_n(f) \defeq \sum_{k=-2^n+1}^{2^n-1}
  \psi_{n,k}(F) \varphi(f(k\cdot 2^{-n})) \psi_{n,k}(F)
\end{equation}
for \(f\colon (-1,1)\to A\) continuous with \(f(\pm1)=0\) defines a
map \(\xi_n\colon S\hot A\to\Pcomp(\Hils_BL)\).  This map is
grading-preserving, completely positive and \(G\)\nb-equivariant
because \(F=F^*\), \(\norm{F}\le1\) and \(F\) is \(G\)\nb-equivariant.
If \(f\ge0\), then
\[
\xi_n(f) \le
\sum_{k=-2^n+1}^{2^n-1} \psi_{n,k}(F) \cdot \norm{f(k\cdot 2^{-n})} \cdot
\psi_{n,k}(F)
\le \norm{f}_\infty \sum_{k=-2^n+1}^{2^n-1} \psi_{n,k}(F)^2
\le \norm{f}_\infty;
\]
thus~\(\xi_n\) is contractive.  If \(\pi\colon \Pcomp(\Hils_BL) \to
\Pcomp(\Hils_BL)/\Comp(\Hils_BL)\) denotes the quotient map, then
\(\pi\circ\xi_n = \Xi\circ\Psi^A_n\) because \(\pi(A)\) and~\(\pi(F)\)
commute.  Now we remove the assumption that~\(A\) is trivially graded:

\begin{lemma}
  \label{lem:cp_am_graded}
  There is a sequence of \(G\times\Z/2\)-equivariant completely
  positive contractive maps \(\xi_n\colon S\hot A\to\Pcomp(\Hils_BL)\)
  with \(\pi\circ\xi_n = \Xi\circ\Psi^A_n\) for \(n\in\N\), even
  if~\(A\) is \(\Z/2\)\nb-graded.
\end{lemma}

\begin{proof}
  We fix \(n\in\N\).  To make the proof of complete positivity easy,
  we directly construct the Stinespring dilation of our map~\(\xi_n\).
  Let
  \[
  \Hilm\defeq \bigoplus_{k=0}^{2^n-1} \bigl(\Hils_BL \oplus (\Hils_BL)^\op\bigr).
  \]
  Let~\(A\) act by \(\varphi\oplus\varphi\circ\gamma\) on each summand
  \(\Hils_BL\oplus (\Hils_BL)^\op\).  Let \(x\colon \Hilm\to\Hilm\) be
  the operator that acts by
  \[
  \begin{pmatrix}
    0&2^{-n}k\\2^{-n}k&0
  \end{pmatrix}
  \]
  on the \(k\)th summand.  This operator is self-adjoint, and it
  graded-commutes with the representation of~\(A\) because we
  take~\(\varphi\gamma\) for the second summands.  Thus the functional
  calculus for~\(x\) provides a \Star{}homomorphism
  \(S\to\Bound(\Hilm)\) that graded-commutes with~\(A\).  Hence we get a
  \(G\times\Z/2\)-equivariant \Star{}homomorphism \(\alpha\colon S\hot
  A\to\Bound(\Hilm)\).  We let \(\xi_n(f)\defeq V^* \alpha(f) V\) for
  all \(f\in S\hot A\), where \(V=(V_k)_{k\in I_n}\colon \Hils_BL \to
  \Hilm\) has the components
  \begin{align*}
    2^{-1/2}(\psi_{n,k}(F)+\psi_{n,-k}(F))&\colon \Hils_BL\to\Hils_BL,\\
    2^{-1/2}(\psi_{n,k}(F)-\psi_{n,-k}(F))&\colon \Hils_BL\to(\Hils_BL)^\op\\
    \shortintertext{for \(k>0\), and}
    \psi_{n,0}(F) = 2^{-1}(\psi_{n,k}(F)+\psi_{n,-k}(F))&\colon \Hils_BL\to\Hils_BL,\\
    0 = 2^{-1}(\psi_{n,k}(F)-\psi_{n,-k}(F))&\colon \Hils_BL\to(\Hils_BL)^\op
  \end{align*}
  for \(k=0\).  Notice that~\(V_k\) is grading-preserving because
  \(\psi_{n,k}+\psi_{n,-k}\) is an even function and
  \(\psi_{n,k}-\psi_{n,-k}\) is an odd function.  Since~\(V\) is
  \(G\)\nb-invariant as well, \(\xi_n\) is
  \(G\times\Z/2\)-equivariant.  The map~\(\xi_n\) is completely
  positive.  Since
  \begin{multline*}
    V^*V = \biggl(\psi_{n,0}^2 + \frac{1}{2}
    \sum_{k=1}^{2^n-1} (\psi_{n,k}+\psi_{n,-k})^2 +
    (\psi_{n,k}-\psi_{n,-k})^2 \biggr)(F)
    \\= \biggl( \psi_{n,0}^2 + \sum_{k=1}^{2^n-1} \psi_{n,k}^2+\psi_{n,-k}^2
    \biggr)(F)
    = \biggl( \sum_{k\in I_n} \psi_{n,k}^2\biggr)(F)
    \le 1,
  \end{multline*}
  the map~\(\xi_n\) is completely contractive.

  Let \(f\in S\) and \(a\in A\).  If \(f\in S\) is even, then
  \begin{align*}
    \xi_n(f\hot a)
    &= \psi_{n,0}(F) f(0)\varphi(a) \psi_{n,0}(F)
    \\&\qquad+ \sum_{k=1}^{2^n-1} (\psi_{n,k}(F)+\psi_{n,-k}(F)) f(2^{-n}k)\varphi(a)
    (\psi_{n,k}(F)+\psi_{n,-k}(F))
    \\&\qquad+ (\psi_{n,k}(F)-\psi_{n,-k}(F)) f(2^{-n}k)\varphi\gamma(a)
    (\psi_{n,k}(F)-\psi_{n,-k}(F));
  \end{align*}
  if \(f\in S\) is odd, then
  \begin{multline*}
    \xi_n(f\hot a)
    = \sum_{k=1}^{2^n-1} (\psi_{n,k}(F)-\psi_{n,-k}(F)) f(2^{-n}k)\varphi(a)
    (\psi_{n,k}(F)+\psi_{n,-k}(F))
    \\+ (\psi_{n,k}(F)+\psi_{n,-k}(F)) f(2^{-n}k)\varphi\gamma(a)
    (\psi_{n,k}(F)-\psi_{n,-k}(F))
  \end{multline*}

  Now we use that~\(\pi(F)\) graded-commutes with~\(\pi\varphi(A)\) to
  simplify \(\pi\circ\xi_n(f\hot a)\).  For even~\(f\), this is equal
  to the \(\pi\)\nb-image of
  \begin{multline*}
    \frac{1}{2}\sum_{k=1}^{2^n-1} (\psi_{n,k}+\psi_{n,-k})^2(F) f(2^{-n}k)\varphi(a)
    + (\psi_{n,k}-\psi_{n,-k})^2(F) f(2^{-n}k)\varphi(a)
    \\ + \psi_{n,0}^2(F) f(0)\varphi(a)
    = \sum_{k\in I_n} \psi_{n,k}^2(F) f(2^{-n}k) \varphi(a)
    = \Psi_n^A(f)(F)\cdot \varphi(a),
  \end{multline*}
  which is \(\Xi\circ\Psi_n^A(f\hot a)\).  For odd~\(f\),
  \(\pi\circ\xi_n(f\hot a)\) is equal to the \(\pi\)\nb-image of
  \begin{align*}
    &\phantom{{}={}} \sum_{k=1}^{2^n-1} (\psi_{n,k}+\psi_{n,-k})(F)(\psi_{n,k}-\psi_{n,-k})(F)
    f(2^{-n}k)\varphi(a)
    \\&= \sum_{k=1}^{2^n-1} (\psi_{n,k}^2-\psi_{n,-k}^2)(F) f(2^{-n}k) \varphi(a)
    \\&= \sum_{k\in I_n} \psi_{n,k}^2(F) f(2^{-n}k) \varphi(a)
    = \Psi_n^A(f)(F)\cdot \varphi(a),
  \end{align*}
  which is \(\Xi\circ\Psi_n^A(f\hot a)\) once again.  Thus
  \(\Xi\circ\Psi^A_n(f\hot a)= \pi(\Psi^A_n(f)(F)\cdot \varphi(a))\) for
  all \(f\in S\), \(a\in A\), as desired.
\end{proof}

Let \(\xi_{n+s} = (1-s)\xi_n+ s\xi_{n+1}\) for \(n\in\N\),
\(s\in[0,1]\).  The maps \((\xi_s)_{s\in L}\) form a continuous family
of grading-preserving, \(G\)\nb-equivariant, completely positive
contractions \(\xi_s\colon S \hot A\to\Pcomp(\Hils_BL)\).  In the
following, we view~\(\xi_s\) as a family of functions
\(\xi_{s,t}\colon S\hot A\to\Comp(\Hils_B)\), and we lift~\(\Xi\) to
an asymptotic morphism \(\bar\Xi_t\colon S\hot A\to\Comp(\Hils_B)\).

\begin{lemma}
  \label{lem:asymptotically_mult}
  For separable~\(A\), there is a continuous increasing function
  \(t_0\colon L\to L\) with \(\lim_{s\to\infty} t_0(s)=\infty\) such
  that for all \(t\ge t_0\), \(\xi_{s,t(s)}\colon S\hot
  A\to\Comp(\Hils_B)\) is asymptotically equal to the
  reparametrisation~\(\Xi_{t(s)}\) of~\(\Xi\) and hence an asymptotic
  morphism in the same class as~\(\Xi\).
\end{lemma}

\begin{proof}
  Since \(S\hot A\) is separable, there is a sequence~\((f_i)\)
  whose closed linear span is \(S \hot A\).  For the
  asymptotic equality we need \(\pi\xi_{s,t(s)}(f_i) =
  \Xi_{t(s)}(f_i)\) for \(i\in\N\).  We have norm convergence
  \(\lim_{s\to\infty}\Psi^A_s(f_i)=f_i\) for all \(i\in\N\).  Since
  \(\norm{f_i}\to0\) and \(\Psi^A_s\) is uniformly bounded, this
  convergence is uniform.  Hence for each \(n\in\N\) there is \(s_n\in
  L\) such that \(\norm{\Psi^A_s(f_i)-f_i}<1/n\) for all \(s\ge s_n\),
  \(i\in\N\).  We may assume that the sequence~\((s_n)\) is strictly
  increasing with \(\lim_{n\to\infty} s_n=\infty\).

  Since \(\pi\circ\xi_s = \Xi\circ\Psi^A_s\) and~\(\Xi\) is a
  \Star{}homomorphism, we get \(\norm{\pi\circ\xi_s(f_i)-\Xi(f_i)}
  <1/n\) for all \(s\ge s_n\), \(i\in\N\).  By definition of the
  quotient norm in \(\Pcomp(\Hils_BL)/\Comp(\Hils_BL)\), we may find
  \(t_i(s,n)\in L\) with \(\norm{\xi_{s,t}(f_i) -
    \bar\Xi_t(f_i)}<1/n\) for \(s\ge s_n\), \(t\ge t_i(s,n)\).  Since
  \(\norm{f_i}\to0\) for \(i\to\infty\), there are only finitely
  many~\(i\) with \(\norm{\xi_{s,t}(f_i)}\ge 1/2n\) and
  \(\bar\Xi_t(f_i)\ge 1/2n\); hence we may find \(t(s,n)\) independent
  of~\(i\) with \(\norm{\xi_{s,t}(f_i) - \bar\Xi_t(f_i)}<1/n\) for all
  \(i\in\N\), \(s\ge s_n\), \(t\ge t(s,n)\).

  Now choose~\(t_0(s)\) increasing and continuous with
  \(\lim_{s\to\infty} t_0(s)=\infty\) and \(t_0(s)\ge t(s,n)\) for
  \(s\in [s_n,s_{n+1}]\).  If \(t(s)\ge t_0(s)\) for all \(s\in L\),
  then \(\norm{\xi_{s,t(s)}(f_i)-\bar\Xi_{t(s)}(f_i)}<1/n\) for all
  \(s\in [s_n,s_{n+1}]\) and all \(i\in\N\).  Thus \(\xi_{s,t(s)}\)
  and~\(\bar\Xi_{t(s)}\) are asymptotically equal.  This implies
  that~\(\xi_{s,t(s)}\) is an asymptotic morphism
  because~\(\bar\Xi_t\) is one.
\end{proof}

The asymptotic morphism~\(\xi_{s,t(s)}\) from \(S \hot A\) to
\(\Comp(\Hils_B)\) in Lemma~\ref{lem:asymptotically_mult} is also
linear, completely positive contractive and
\(G\times\Z/2\)\nb-equivariant.  Thomsen~\cite{Thomsen:Asymptotic_KK}
describes \(\KK^G(A,B)\) using asymptotic homomorphisms with these
extra properties.  We cannot directly appeal
to~\cite{Thomsen:Asymptotic_KK} because we have replaced the ungraded
suspension on both \(A\) and~\(B\) by the graded suspension~\(S\)
on~\(A\) alone.  It is well-known, however, that both approaches give
the same definition of equivariant \(E\)\nb-theory.  For the same
reason, both approaches with added complete positivity requirements
give \(\KK^G(A,B)\).  Let us make this more explicit.

An asymptotic morphism~\((\xi_t)\) from \(S \hot A\) to
\(\Comp(\Hils_B)\) gives an extension
\[
0 \to \Cont_0(L,\Comp(\Hils_B)) \to E \to
S \hot A \to 0,
\]
where \(E=\Cont_0(L,\Comp(\Hils_B))+\xi(S\hot A)\); it comes with
evaluation homomorphisms \(\epsilon_t\colon E\to\Comp(\Hils_B)\) for
\(t\in L\).  If the asymptotic morphism is
\(G\times\Z/2\)\nb-equivariant, completely positive and contractive,
then the extension above has a \(G\times\Z/2\)\nb-equivariant,
completely positive and contractive cross-section.  Hence there is a
long exact sequence in \(\KK^{G\times\Z/2}\) for this extension.
Since the kernel is contractible, we get that the quotient map in the
extension is invertible in \(\KK^{G\times\Z/2}\).  Composing its
inverse with the evaluation homomorphism, we get a class in
\[
\KK^{G\times\Z/2}(S \hot A,\Comp(\Hils_B))
\cong \KK^{G\times\Z/2}(S \hot A,B)
\cong \KK^G(A,B).
\]
Here we use a description of \(\KK^G\) for \(\Z/2\)-graded
\(\Cst\)\nb-algebras in terms of \(G\times\Z/2\)-equivariant Kasparov
theory that goes back to Haag in the non-equivariant case and is
extended to the equivariant case in~\cite{Meyer:Equivariant}.

Thus we attach a class in \(\KK^G(A,B)\) to a cycle for
\(\KE^G(A,B)\).  Since the same construction applies to homotopies,
this construction descends to a well-defined map \(\xi\colon
\KE^G(A,B)\to\KK^G(A,B)\).  By design, the composite map
\[
\KE^G(A,B)\to\KK^G(A,B)\to\E^G(A,B)
\]
is the functor~\(\Xi\) of~\cite{Dumitrascu:KE}.

The Kasparov product in~\(\KK^G\) becomes the composition of
completely positive equivariant asymptotic morphisms in the above
picture.  A composite of two completely positive equivariant
asymptotic morphisms is again completely positive and equivariant.
So the same argument as in~\cite{Dumitrascu:Thesis} shows
that~\(\xi\) is a functor.

\begin{proposition}
  The composite map
  \[
  \KK^G(A,B)\to\KE^G(A,B)\to\KK^G(A,B)
  \]
  is the identity on \(\KK^G(A,B)\).
\end{proposition}

\begin{proof}
  This clearly holds on the class in \(\KK^G(A,B)\) of a
  grading-preserving equivariant \Star{}homomorphism \(f\colon S\hot
  A\to B\).  If this~\(f\) is a \(\KK^G\)-equivalence,
  then~\([f]^{-1}\) is mapped to~\([f]^{-1}\) as well by
  functoriality.  Hence any composite of such classes is mapped to
  itself by functoriality.  Any class in \(\KK^G\) may be written as
  such a composition of classes of \([f]\) and~\([f]^{-1}\).  This
  follows from the Cuntz picture for \(\KK^G(A,B)\cong
  \KK^{G\times\Z/2}(S\hot A,B)\) in~\cite{Meyer:Equivariant}.
\end{proof}

\setlength{\emergencystretch}{3em}


\begin{bibdiv}
  \begin{biblist}
\bib{Connes-Higson:Deformations}{article}{
  author={Connes, Alain},
  author={Higson, Nigel},
  title={D\'eformations, morphismes asymptotiques et $K$\nobreakdash -th\'eorie bivariante},
  journal={C. R. Acad. Sci. Paris S\'er. I Math.},
  volume={311},
  date={1990},
  number={2},
  pages={101--106},
  issn={0764-4442},
  review={\MRref {1065438}{91m:46114}},
}

\bib{Dumitrascu:Thesis}{thesis}{
  author={Dumitra\c {s}cu, Dorin},
  title={A new approach to bivariant K\nobreakdash -theory},
  date={2001},
  institution={Pennsylvania State University},
  type={phdthesis},
  eprint={https://etda.libraries.psu.edu/paper/5876/},
}

\bib{Dumitrascu:KE}{article}{
  author={Dumitra\c {s}cu, Dorin},
  title={On an intermediate bivariant $K$\nobreakdash -theory for $C^*$\nobreakdash -algebras},
  date={2015},
  status={preprint},
}

\bib{Guentner-Higson-Trout:Equivariant_E}{article}{
  author={Guentner, Erik P.},
  author={Higson, Nigel},
  author={Trout, Jody},
  title={Equivariant $E$\nobreakdash -theory for $C^*$\nobreakdash -algebras},
  journal={Mem. Amer. Math. Soc.},
  volume={148},
  date={2000},
  number={703},
  pages={viii+86},
  issn={0065-9266},
  review={\MRref {1711324}{2001c:46124}},
  doi={10.1090/memo/0703},
}

\bib{Kasparov:Operator_K}{article}{
  author={Kasparov, Gennadi G.},
  title={The operator \(K\)\nobreakdash -functor and extensions of \(C^*\)\nobreakdash -algebras},
  journal={Izv. Akad. Nauk SSSR Ser. Mat.},
  volume={44},
  date={1980},
  number={3},
  pages={571--636, 719},
  issn={0373-2436},
  translation={ journal={Math. USSR-Izv.}, volume={16}, date={1981}, number={3}, pages={513--572}, doi={10.1070/IM1981v016n03ABEH001320}, },
  review={\MRref {582160}{81m:58075}},
  eprint={http://mi.mathnet.ru/izv1739},
}

\bib{Kasparov:Novikov}{article}{
  author={Kasparov, Gennadi G.},
  title={Equivariant \(KK\)-theory and the Novikov conjecture},
  journal={Invent. Math.},
  volume={91},
  date={1988},
  number={1},
  pages={147--201},
  issn={0020-9910},
  review={\MRref {918241}{88j:58123}},
  doi={10.1007/BF01404917},
}

\bib{Meyer:Equivariant}{article}{
  author={Meyer, Ralf},
  title={Equivariant Kasparov theory and generalized homomorphisms},
  journal={\(K\)\nobreakdash -Theory},
  volume={21},
  date={2000},
  number={3},
  pages={201--228},
  issn={0920-3036},
  review={\MRref {1803228}{2001m:19013}},
  doi={10.1023/A:1026536332122},
}

\bib{Thomsen:Asymptotic_KK}{article}{
  author={Thomsen, Klaus},
  title={Asymptotic homomorphisms and equivariant $KK$-theory},
  journal={J. Funct. Anal.},
  volume={163},
  date={1999},
  number={2},
  pages={324--343},
  issn={0022-1236},
  review={\MRref {1680467}{2000c:19008}},
  doi={10.1006/jfan.1998.3377},
}


  \end{biblist}
\end{bibdiv}
\end{document}